\documentclass[11pt,leqno]{amsart}

\usepackage[top=2.4 cm,bottom=1.9cm,left=2.5 cm,right=2 cm]{geometry}
\usepackage[utf8]{inputenc}

\usepackage{bbm}
\usepackage{esint}
\usepackage{graphicx}
\usepackage{color}
\usepackage{amsfonts}
\usepackage{psfrag}
\usepackage{enumitem}
\newcounter{stepnb}

\usepackage{hyperref}\hypersetup{colorlinks=true,linkcolor=blue,citecolor=blue}

\newtheorem{theorem}{Theorem}
\newtheorem{lemma}[theorem]{Lemma}
\newtheorem{proposition}[theorem]{Proposition}

\newtheorem{conj}[theorem]{Conjecture}

\theoremstyle{plain} \newtheorem{corollary}[theorem]{Corollary}

\theoremstyle{remark}
\newtheorem{remark}[theorem]{Remark}
\allowdisplaybreaks
\theoremstyle{plain} \newtheorem{assu}[theorem]{Assumption}

\newcommand{\N}{\mathbb{N}}

\newcommand{\R}{\mathbb{R}}

\newcommand{\D}{{\mathcal D}}

\newcommand{\ee}{\varepsilon}
\newcommand{\e}{\varepsilon}

\newcommand{\TV}{\text{\rm TotVar}}
\newcommand{\Lip}{\mathrm{Lip}}

\newcommand{\be}{\begin{equation}}
\newcommand{\eq}{\end{equation}}
\newcommand{\eps}{\varepsilon}

\newcommand{\loc}{\mathrm{loc}}

\begin{document}

\title[Local limit of nonlocal traffic models]{Local limit of nonlocal traffic models:\\ convergence results and total variation blow-up}

\author[M.~Colombo]{Maria Colombo}
\address{M.C. EPFL SB, Station 8, 
CH-1015 Lausanne, Switzerland
}
\email{maria.colombo@epfl.ch}
\author[G.~Crippa]{Gianluca Crippa}
\address{G.C. Departement Mathematik und Informatik,
Universit\"at Basel, Spiegelgasse 1, CH-4051 Basel, Switzerland.}
\email{gianluca.crippa@unibas.ch}
\author[E.~Marconi]{Elio Marconi}
\address{E.M. Departement Mathematik und Informatik,
Universit\"at Basel, Spiegelgasse 1, CH-4051 Basel, Switzerland.}
\email{elio.marconi@unibas.ch}

\author[L.~V.~Spinolo]{Laura V.~Spinolo}
\address{L.V.S. IMATI-CNR, via Ferrata 5, I-27100 Pavia, Italy.}
\email{spinolo@imati.cnr.it}
\maketitle
{
\rightskip .85 cm
\leftskip .85 cm
\parindent 0 pt
\begin{footnotesize}

\noindent
{\sc Abstract.} Consider a nonlocal conservation where the flux function depends on the convolution of the solution with a given kernel.   In the singular local limit obtained by letting the convolution kernel converge to the Dirac delta one formally recovers a conservation law. However, recent counter-examples show that in general the solutions of the nonlocal equations do not converge to a solution of the conservation law. 
In this work we focus on nonlocal conservation laws modeling vehicular traffic: in this case, the convolution kernel is anisotropic. We show that, under fairly general assumptions on the (anisotropic) convolution kernel, the nonlocal-to-local limit can be rigorously justified provided the initial datum satisfies a one-sided Lipschitz condition and is bounded away from $0$. We also exhibit a counter-example showing that, if the initial datum attains the value $0$, then there are severe obstructions to a convergence proof. 

\medskip\noindent
{\sc Keywords:} traffic model, nonlocal conservation law, anisotropic kernel, nonlocal continuity equation, singular limit, local limit, Ole\u{\i}nik estimate. 

\medskip\noindent
{\sc MSC (2010):} 35L65.

\end{footnotesize}
}
\section{Introduction}
We deal with the nonlocal conservation law (or nonlocal continuity equation)
\begin{equation}
\label{e:nl}
     \partial_t u + \partial_x \Big[ u V ( u \ast \eta) \Big]=0, 
\end{equation}
where $u: \R_+ \times \R \to \R$ is the unknown, $V : \R \to \R$ is a given Lipschitz continuous function and in the nonlocal term the symbol $\ast$ denotes the convolution with respect to the space variable only. The convolution kernel $\eta \in L^1(\R)$ is compactly supported, nonnegative, and has unit integral.  
Conservation laws involving nonlocal terms appear in models for sedimentation~\cite{Sedimentation}, pedestrian crowds~\cite{ColomboGaravelloMercier,ColomboHertyMercier}, vehicular traffic~\cite{BlandinGoatin,GoatinScialanga}, and others.

\medskip
 
In the present work we are concerned with the nonlocal-to-local limit. More precisely, consider a parameter~$\ee>0$, define $\eta_\ee$ by setting $\eta_\ee(x): = \eta(x/\ee) / \ee$, and consider the family of nonlocal equations
\begin{equation}
\label{e:nle}
     \partial_t u_\ee + \partial_x \Big[ u_\ee V ( u_\ee \ast \eta_\ee) \Big]=0 
\end{equation}
which are obtained from~\eqref{e:nl} by replacing $\eta$ with $\eta_\ee$. 
 When $\ee \to 0^+$, the kernel $\eta_\ee$ converges weakly$-\ast$ in the sense of measures to the Dirac delta, and hence one formally recovers the (local) conservation law 
 \be \label{e:cl}
     \partial_tu + \partial_x \Big[ u V ( u ) \Big]=0.
 \eq
In~\cite{ACT} Amorim, R. Colombo and Teixeira posed the following question: can we rigorously justify this formal limit? In other words, can we show that when $\ee \to 0^+$ the solution $u_\ee$ of~\eqref{e:nl} converges to the
entropy admissible solution of~\eqref{e:cl}? In a previous work~\cite{ColomboCrippaSpinolo}, counter-examples were exhibited showing that the answer to this question is, in general, negative. See also \cite{ColomboCrippaGraffSpinolo} for the role of numerical viscosity.

However, the results in~\cite{ColomboCrippaSpinolo} do not rule out the possibility that, in some more specific case, convergence indeed holds. In particular, several recent works (see for instance Blandin and Goatin~\cite{BlandinGoatin} and Chiarello and Goatin~\cite{ChiarelloGoatin})
 have been devoted to the analysis of the case when $V$ is monotone nonincreasing, the initial datum is nonnegative, and the convolution kernel is anisotropic, in particular it is supported on the negative axis $]- \infty, 0]$. 
This case is very relevant in the modelling of vehicular traffic, where the unknown $u$ represents the density of cars and $V$ their speed. Assuming that $V$ is monotone nonincreasing is standard in local and nonlocal traffic models: the higher the density of cars on a road, the lower their speed. The assumption that the convolution kernel is supported on the negative axis expresses the fact that one expects the drivers to decide their speed based only on the downstream traffic density, i.e.~they only look forward, not backward. 

Remarkably, when $V$ is is monotone nonincreasing and the convolution kernel is supported on the negative axis $]-\infty, 0]$, stronger analytic results are available. More precisely:
\begin{itemize}
\item The nonlocal equation~\eqref{e:nl} satisfies a maximum principle, see~\cite[Theorem~1]{BlandinGoatin} (see also Proposition \ref{T_BS} below). 
\item The nonlocal equation~\eqref{e:nl} is monotonicity preserving, that is if the initial datum is bounded and monotone, so is $u(t, \cdot)$ for every $t>0$, see~\cite[Proposition~2]{BlandinGoatin}. This allows to show that, if the initial datum is monotone and bounded, the nonlocal-to-local limit can be rigorously justified under suitable assumptions on the function $V$, see~\cite{KeimerPflug2}.  
\item Very recently, Bressan and Shen \cite{BS_nonlocal} proved that, if the convolution kernel is $\eta(x) = 1_{]-\infty, 0]} e^{-x}$ and the initial datum is bounded away from $0$ and has bounded total variation, then the solutions of $u_\eps$ of \eqref{e:nle} converge to a weak solution of \eqref{e:cl}. Under the further assumption that the function $V$ is affine,  
they also show that that the limit is the unique entropy admissible solution. The analysis in \cite{BS_nonlocal} relies on a change of variable which allows to rewrite \eqref{e:nle} as  a $2 \times 2$ system of conservation laws with relaxation, provided the convolution kernel is exactly $\eta(x) = 1_{]-\infty, 0]} e^{-x}$.
\item To conclude, we point out  that the numerical experiments in~\cite{ACT,BlandinGoatin} suggest that in the case of anisotropic kernels the behavior of the solutions $u_\ee$ in the local limit $\ee \to 0^+$ is more stable than in the case of general convolution kernels. In particular, they suggest that, if $V$ is monotone nonicreasing and the convolution kernel is supported on $]-\infty, 0]$,  the total variation $\TV \, u_\ee (t, \cdot)$ is a monotone nonincreasing function of time.  
\end{itemize}
Our main positive results establishes the nonlocal-to-local limit from~\eqref{e:nle} to the entropy admissible solution of~\eqref{e:cl} under fairly general assumptions 
on $V$ and on the (anisotropic) convolution kernel, provided that the initial datum has bounded total variation, is bounded away from $0$ and satisfies a one-sided Lipschitz condition. Note that our assumptions on $\eta$ and $V$ are much weaker than those in~\cite{BS_nonlocal}, but on the other hand we impose stronger assumptions on the initial datum, more precisely we have to assume that it satisfies a one-sided Lipschitz condition that is defined in the following.   

To rigorously state our result we have to introduce some notation. First,
we introduce the assumptions we impose on $V$ and $\eta$.
\begin{assu} \label{A_1}
The function $V$ is of class $C^2$ and satisfies $V''\leq 0$. Also, there are $\delta,u_{\max}>0$ such that 
\begin{equation} \label{e:comeV}
 V(u_{\max})=0 \qquad \mbox{and} \qquad V'(v)\le -\delta \quad \text{ for every $v \in [0,u_{\max}]$}.
\end{equation}
\end{assu}
Note that the above assumption is fairly common in traffic models: $u_{\max}$ represents the maximum possible car density, which occurs when cars are completely packed and cannot move. 
\begin{assu}\label{A_2}
The convolution kernel $\eta$ satisfies 
\be \label{e:staneta} 
    \eta (x) \ge 0 \; \text{for every $x \in \R$}, \quad  \eta (x) = 0 \; \text{for every $x \in ]0, + \infty]$}, \quad 
    \int_\R \eta (x) dx =1. 
\eq
Also, $\eta$ is Lipschitz continuous on $]-\infty, 0]$ and there is a constant $D >0$ such that 
\begin{equation} \label{e:D}
\eta(y)\le D \eta'(y), \quad \text{for 
a.e. $y\in ]-\infty,0[$}.
\end{equation}
\end{assu}
In the following we focus on the Cauchy problem, so we impose the initial condition
\begin{equation}\label{E_initial}
u(0,\cdot)=u_0. 
\end{equation} 
We assume that the initial datum belongs to the set (the same as in~\cite{BS_nonlocal})
\begin{equation}\label{eqn:D}
\D:=\big\{ u_0 \in L^\infty(\R): \TV(u_0)<\infty, \, u_0(x)\in [0, u_{\max}] \, \mbox{ for a.e. }x\in \R \big\}.
\end{equation}
Note that the assumption $u_0(x)\in [0, u_{\max}] $ models the fact that 
the initial density $u_0$ should be positive and not exceed the maximum possible density. 
To conclude, for every $f:\R\to \R$, we define the quantity $\Lip^-f $ by setting 
\begin{equation*}
\Lip^-f:= -\inf_{x<y}\frac{f(y)-f(x)}{y-x}.
\end{equation*}
The quantity above bounds the negative part of the difference quotients. In particular $\Lip^-f<\infty$ implies that $f$ has no jumps with negative sign, while positive jumps are allowed.
Our main result of this section is Theorem~\ref{P_one_side}, which establishes a new uniform decay on the negative part of the space derivative of $u_\eps$, that is on $\Lip^-u_\eps(t)$.
\begin{theorem}\label{P_one_side}
Let $V$ and $\eta$ satisfy Assumptions \ref{A_1} and \ref{A_2}, respectively. Assume moreover that $u_0\in \D$ satisfies $\inf u_0 >0$ and $\Lip^-u_0\le L$ for some $L> 0$ .
Let $u_\eps(t)
$ be the solution of the Cauchy problem \eqref{e:nle},\eqref{E_initial}. 
If
\begin{equation} \label{e:ee}
\e < \frac{\inf u_0}{2 D L},
\end{equation}
where $D>0$ is the same as in Assumption \ref{A_1}, 
then 
\begin{equation}\label{E_Lip}
\Lip^-u_\eps(t, \cdot) \le \frac{L}{2\delta Lt+1}<\frac{1}{2\delta t}, \quad \text{for every $t\ge 0$}.
\end{equation}
\end{theorem}
Some remarks are here in order: first, owing to Assumption~\ref{A_1} the flux function $u\mapsto uV(u)$ satisfies $(u V(u))'' \leq - 2 \delta$ and hence the decay estimate~\eqref{E_Lip} is consistent with the celebrated Ole\u{\i}nik estimate~\cite{Oleinik} for (local) conservation laws~\eqref{e:cl}. Second, as a consequence of the decay estimate	~\eqref{E_Lip}  we rigorously establish the nonlocal-to-local limit, more precisely we show that the solutions of  the nonlocal Cauchy problems \eqref{e:nle},\eqref{E_initial} converge to the
entropy admissible solution of \eqref{e:cl},\eqref{E_initial} strongly in $L^1_{\mathrm{loc}}(\R^+ \times \R)$ as $\e\to 0$. Here is the precise statement.  
\begin{corollary}\label{cor:convergence}
Assume that $V$ and $\eta$ satisfy Assumptions \ref{A_1} and \ref{A_2}, respectively, and that $u_0\in \D$ satisfies  $\Lip^-u_0<\infty$ and $\inf_{x\in \R} u_0>0$. Let $u_\e$ be the solution of the Cauchy problem \eqref{e:nle}, \eqref{E_initial}. 
Then, for every $t\ge 0$, the family $u_\e(t, \cdot)$ strongly converges  in $L^1_\loc(\R)$  as $\e \to 0^+$ to the entropy admissible solution of the Cauchy problem \eqref{e:cl}, \eqref{E_initial} $u(t, \cdot)$.
\end{corollary}

\begin{remark}
The assumptions on $u_0$ imposed in the statement of Corollary~\ref{cor:convergence}, namely $\Lip^-u_0<\infty$ and $\inf_{x\in \R} u_0>0$, can be relaxed to $0 \leq u_0 \leq u_{\rm max}$ if one allows
for an $\ee$-dependence of the initial datum of the Cauchy problem  \eqref{e:nle},\eqref{E_initial}, that is 
 if one replaces the condition $u_\eps(0,\cdot)=u_0$ with the condition $u_\eps(0,\cdot)=u_{0, \eps}$ for a suitably chosen sequence $u_{0, \eps} \to u_0$. For instance, under the sole assumption $0 \leq u_0 \leq u_{\rm max}$, 
one can consider $ u_{0,\eps} = \min \{ u_0 \ast \rho_{ \eps^{2/3} }, c_0 \eps^{1/3} \}$, where $\rho_{ \nu}(x):= \rho(x/\nu)/\nu$ and $\rho: \R \to \R^+$ is a fixed smooth convolution kernel.
In this case we can rigorously establish the same nonlocal-to-local limit as in the statement of Corollary~\ref{cor:convergence}. The proof relies again on Theorem~\ref{P_one_side}, which also in this case provides a uniform bound on $\Lip^-u_\eps(t, \cdot)$ as $\eps \to 0$.
\end{remark}
\medskip

We now discuss our main negative result concerning the nonlocal-to-local limit from
 \eqref{e:nle} to  \eqref{e:cl}. First, we point out that the proof of Corollary~\ref{cor:convergence} relies on the Helly-Kolmogorov Compactness Theorem. More precisely, we show that the one-sided Lipschitz estimate~\eqref{E_Lip} implies a uniform local bound on the total variation, i.e. it implies that $\TV \{ u_\ee (t, \cdot); ]-R, R[ \}$ is uniformly bounded with respect to $t$ and $\ee$, for every $R>0$, see~\eqref{e:tv}. As a matter of fact, to the best of our knowledge \emph{all} the known convergence results on the  nonlocal-to-local limit (that is, Corollary~\ref{cor:convergence} and the results in~\cite{BS_nonlocal,KeimerPflug2}) are based on the Helly-Kolmogorov Compactness Theorem and require a uniform control on the total variation. The only exception is the convergence result due to Zumbrun~\cite{Zumbrun}, which however only applies to time intervals where the solution of the conservation law~\eqref{e:cl},\eqref{E_initial} is very regular (of class $C^4$). 
We point out in passing that the maximum principle implies weak-$\ast$ compactness of the family $\{ u_\ee \}$, but weak-$\ast$ convergence alone does not allow to pass to the limit in the equation~\eqref{e:nle}. Note furthermore that the semigroup of entropy admissible solutions of scalar conservation laws~\eqref{e:cl} is total variation decreasing, and hence uniform bounds on the total variation of~\eqref{e:nle} are somehow natural in view of a convergence result. Also, numerical experiments in~\cite{ACT,BlandinGoatin} suggest that, in the case of anisotropic convolution kernels, the total variation of~\eqref{e:nle}  is uniformly bounded. 

Our main negative result states that, if $V$, $\eta$ and $u_0$ satisfy \emph{all} the hypotheses of Corollary~\ref{cor:convergence} \emph{but} the condition $\inf u_0 >0$, then the total variation of~\eqref{e:nle},\eqref{E_initial} can blow up in $\ee$ for every positive time. Althought strictly speaking this does not rule out convergence in the nonlocal-to-local limit, it provides a severe obstruction to a convergence proof, as it prevents the application of the argument used in the proof of basically all the known convergence results. 
\begin{theorem} 
\label{t:blowup}
Assume that $V(u) = 1-u$ and that either $\eta (x): = \mathbbm{1}_{[-1, 0]}(x)$ or $\eta$ satisfies the 
following assumption: $\eta$ satisfies~\eqref{e:staneta}, it is Lipschitz continuous on $]- \infty, 0]$ and 
$\eta'(x) \ge 0$ for a.e. $x \in ]- \infty, 0]$.  Then there is $u_0 \in L^\infty(\R)$ such that $0 \leq u_0 (x) \leq 1$ for a.e.~$x \in \R$,  $\mathrm{TotVar} \ u_0 < + \infty$, $\mathrm{Lip}^- u_0 < + \infty$ 
and the solution of the Cauchy problem~\eqref{e:nle},\eqref{E_initial} satisfies
\begin{equation} \label{e:blowup}
   \sup_{\ee >0} \mathrm{TotVar} \ u_\ee (\tau, \cdot) = + \infty, \quad \text{for every $\tau>0$}. 
\end{equation}
\end{theorem}
Some remarks are here in order. First, $V(u) = 1-u$ satisfies Assumption~\ref{A_1}. Second, our counter-example is completely explicit and in~\S\ref{s:pbu} we provide the precise formula for an initial datum $u_0$ satisfying the statement of the theorem,  see~\eqref{e:uzero} and~\eqref{e:newa}. Third, if the initial datum $u_0$ were monotone one could apply the results in~\cite{BlandinGoatin,KeimerPflug2} and establish uniform bounds on the total variation, and indeed the initial datum we exhibit in the proof of Theorem~\ref{t:blowup} is not monotone. Fourth, it is natural to compare 
Corollary~\ref{cor:convergence} and Theorem~\ref{t:blowup} and wonder what are the sharp conditions that prevent the total variation blow up. Our guess is that the conditions $\mathrm{Lip}^- u_0<+\infty$ is just a technical hypothesis, and that the key condition to obtain a uniform bound on the total variation is $\inf u_0 >0$. More precisely, we propose the following conjecture.
\begin{conj}\label{C_conj}
Assume that $V$ and $\eta$ satisfy Assumptions \ref{A_1} and \ref{A_2}, respectively, and that $u_0\in \D$ satisfies   $\inf_{x\in \R} u_0>0$. Let $u_\e$ be the solution of the Cauchy problem \eqref{e:nle}, \eqref{E_initial}. Then for every $T, R>0$ there is a constant $C>0$, possibly depending on $T$, $R$ and $u_0$, such that 
\be \label{e:conj}
     \TV \{ u_\ee (t, \cdot); ]-R, R[ \} \leq C, \quad \text{for every $t \in [0, T]$ and every $\e>0$}.
\eq
\end{conj}
As mentioned above we remark that Conjecture \ref{C_conj} has been proved in \cite{BS_nonlocal} under the assumption that $\eta(x)= 1_{]-\infty, 0]} e^{-x}$.
The exposition is organized as follows: in~\S\ref{sec:prelim} we establish the well-posedness of the Cauchy problem~\eqref{e:nle},\eqref{E_initial} by slightly extending previous results in~\cite{BlandinGoatin,BS_nonlocal,ChiarelloGoatin,KeimerPflug}. In~\S\ref{s:pp} we establish the proof of Theorem~\ref{P_one_side} and Corollary~\ref{cor:convergence} and in~\S\ref{s:pbu} we establish the proof of Theorem~\ref{t:blowup}.

\section{Well-posedness of the Cauchy problem for fixed $\eps>0$}\label{sec:prelim}
In the case of anisotropic kernels, well-posedness of the Cauchy problem \eqref{e:nl},\eqref{E_initial} is  discussed in several works, see for instance \cite{BlandinGoatin,BS_nonlocal,ChiarelloGoatin,KeimerPflug}. The following proposition slightly extends previous well-posedness results. 
\begin{proposition}\label{T_BS}
Let Assumptions \ref{A_1} and \ref{A_2} hold true and fix $\eps>0$. Then there is  a unique semigroup $S^\eps:[0,+\infty[ \times \D \to \D$, continuous in $L^1_\loc$, such that each trajectory 
$t\mapsto S^\eps_t u_0$ provides a distributional solution of the Cauchy problem \eqref{e:nl}, \eqref{E_initial}.

Moreover, the semigroup $S^\eps$ satisfies the following properties.
\begin{itemize}
\item[i)] Assume  $u_0(x) \in [a,b]$ for a.e. $x\in \R$ and for some $0\le a<b\le u_{\max}$. Then
\begin{equation}\label{E_MP}
a\le S_tu_0 (x)\le b \qquad \mbox{for every $t>0$ and for a.e. $x\in \R$}.
\end{equation} 
\item[ii)] For every $k\in \N$ and every $T, A \ge 0$ there is a constant $C=C(k,T,\eta, \e, A)$ such that if $\|u_0\|_{C^k}\le A$, then $\|S_t u_0\|_{C^k}\le C$ for every $t \in [0, T]$.
\item[iii)] Assume that $\|u_0\|_{C^1}<\infty$, then the map $t\mapsto \Lip^-(S^\ee_t u_0)$ is a locally Lipschitz continuous function from $[0, + \infty[$ to $[0,+\infty[$.  
\end{itemize}
\end{proposition}
\begin{proof}
The maximum principle~\eqref{E_MP} is established in~\cite{BlandinGoatin}. The proof of property ii) is provided in~ \cite[\S2]{BS_nonlocal}. We are left with establishing property iii):  we fix $\ee>0$ and by combining property ii) with the equation~\eqref{e:nle}, we establish $C^0$ bounds on $\partial_t u_\eps$. We conclude that $u_\eps$ is of class $C^1$ with respect to both space and time and it is a classical solution of~\eqref{e:nl}. Next, we set $v_\ee:= \partial_x u_\ee$ and we point out that 
\be  \label{e:Lipreg}
    \Lip^-(S^\ee_t u_0) = -\inf_{x \in \R}  \partial_x u^\ee (t, x)= -\inf_{x \in \R}  v^\ee (t, x)   . 
\eq
We use the characteristic lines of~\eqref{e:nl} and we denote by $X_\ee (\cdot, \bar t, \bar x)$ the solution of the Cauchy problem 
\be 
\left\{
\begin{array}{ll}
         \displaystyle{ \frac{d X_\ee}{d t} = V (u_\ee *\eta_\e)(t, X_\ee)} \\ 
         X_\ee (t = \bar t) = \bar x.  \displaystyle{\phantom{\int}}
\end{array}
\right.
\eq
By differentiating~\eqref{e:nl} with respect to $x$ we infer that the material derivative of $v_\ee$ is given by
\begin{equation}\label{E_v_nu}
\partial_t v_\ee + V(u_\ee *\eta_\e)\partial_x v_\ee = - 2v_\ee V'(u_\ee*\eta_\e)v_\ee*\eta_\e - u_\ee V''(u_\ee*\eta_\e)(v_\ee*\eta_\e)^2 - u_\ee V'(u_\ee*\eta_\e)(\partial_x v_\ee *\eta_\e).
\end{equation}
Since $\|u_\ee(t)\|_{C^1}$ is bounded on $ [0,T]$, then the source at the right hand side of \eqref{E_v_nu} is uniformly bounded on $[0,T]\times \R$. 
To conclude, we fix $t_1, t_2 \in [0, T]$ and just to fix the ideas we assume that $ \Lip^-(S^\ee_{t_1} u_0)\ge
 \Lip^-(S^\ee_{t_2} u_0)$. We fix an arbitrarily small constant $h>0$ and a point $x_1$ such that  
\be \label{e:icsuno}
    v_\ee (t_1, x_1) \leq \inf_{x \in \R}  v^\ee (t_1, x) + h.
\eq
By recalling~\eqref{e:Lipreg} we get 
\begin{equation*}
\begin{split}
    | \Lip^-(S^\ee_{t_1} u_0)-
 \Lip^-(S^\ee_{t_2} u_0)| & = 
   \Lip^-(S^\ee_{t_1} u_0)-
 \Lip^-(S^\ee_{t_2} u_0) \stackrel{\eqref{e:Lipreg}}{\leq}
     -  \inf_{x \in \R}  v^\ee (t_1, x) - \Lip^-(S^\ee_{t_2} u_0)\\
& \stackrel{\eqref{e:icsuno}}{\leq}
    -   v_\ee (t_1, x_1) + h  - \Lip^-(S^\ee_{t_2} u_0)
    \stackrel{\eqref{e:Lipreg}}{\leq}  -   v_\ee (t_1, x_1) + h + v_\ee (t_1, X_\ee (t_2, t_1, x_1)). 
\end{split}
\end{equation*}
By using the fact the material derivative~\eqref{E_v_nu} is uniformly bounded  and the arbitrariness of the constant $h$ we conclude that the map $t\mapsto \Lip^-(S^\ee_t u_0)$ is a Lipschitz continuous function on $[0, T]$. 
\end{proof}
\begin{remark}[Preservation of the monotonicity]
\label{r:pm}
For sake of completeness, we sketch here a formal proof of the preservation of the monotonicity of the initial datum rigorously shown in~\cite[Proposition 2]{BlandinGoatin} and~\cite[\S4]{KeimerPflug2}.  
Let us assume to fix the ideas that the initial datum $u_\ee(0, \cdot)$ is nondecreasing, that is (using the same notation as in the proof of Proposition~\ref{T_BS}) $v_\ee(0,\cdot)\geq 0$.
We evaluate~\eqref{E_v_nu} at a minimum point $\bar x$ of $v_\ee(t,\cdot)$ at which $v_\ee(t,\bar x)=0$: we have $\partial_x v_\ee(t,\bar x) = 0$. By using Assumption~\ref{A_2} and integrating by parts we get 
$$
    \partial_x v_\ee *\eta_\e (t,\bar x) = \int_0^\infty \eta_\ee (y) \partial_x v_\ee (t, \bar x-y) dy  = 
    \underbrace{- \eta_\ee (0) \partial_x v_\ee(t,\bar x)}_{=0} + \int_0^\infty \eta'_\ee (y) v_\ee (t, \bar x-y)  dy \ge 0.
$$
We conclude that $\partial_t v_\ee(t,\bar x) \geq 0$ and this yields the preservation of the monotonicity of $u_\ee$.
\end{remark}

\section{Proof of Theorem~\ref{P_one_side} and of Corollary~\ref{cor:convergence}} \label{s:pp}

\subsection{Proof of Theorem~\ref{P_one_side}}
First, we point out that it suffices to establish the statement of Theorem~\ref{P_one_side} under the additional assumption that  $u_0 \in \D$ satisfies $\|u_0\|_{C^2}<\infty$.
Indeed, estimate~\eqref{E_Lip} in the general case $u_0\in \D$ follows by the $L^1_\loc$-continuity of the semigroup $S^\ee_t$ defined in the statement of Proposition~\ref{T_BS} and by the lower semicontinuity of the map $u\mapsto \Lip^-u$.

We fix $u_0 \in \D$ such that $\|u_0\|_{C^2}<\infty$. By arguing as in the proof of Proposition~\ref{T_BS} we infer that $S^\ee_t u_0$ is $C^2$ with respect to space and time and it is a classical solution of~\eqref{e:nl},\eqref{E_initial}.  Next, we fix $T>0$  and we separately consider the following two cases:
\begin{enumerate}
\item[1.] for every $t \in [0,T]$ there is $x\in \R$ such that $\partial_x u_\eps(t,x) \le 0$;
\item[2.] there is $t \in [0,T]$ such that $\partial_x u_\eps(t,x)>0$ for every $x\in \R$.
\end{enumerate}
{\sc Case 1.} Fix $t \in [0, T]$: by using the fact that there is   $x\in \R$ such that $\partial_x u_\eps(t,x) \le 0$ and recalling that $\TV(u_\eps(t))$ and $\|u_\eps(t)\|_{C_1}$ are both finite  owing to Proposition~\ref{T_BS}, we conclude that there is $\bar x\in \R$ such that 
\begin{equation} \label{e:c}
\partial_x u_\eps(t,\bar x)= \min_{x\in \R}\partial_xu_\eps(t, x)= \min_{x\in \R}v_\eps (t, x) :=-c(t) \leq 0,
\end{equation}
where we have used the notation $v_\eps=\partial_x u_\eps$ and to simplify the exposition we write $c(t)$ instead of $c_\ee (t)$. Next, we evaluate~\eqref{E_v_nu} at $(t,\bar x)$ and recall that $V$ is concave and that $u_\ee \ge 0$.  We obtain
\begin{equation}\label{E_lower_der}
\partial_t v_\eps (t,\bar x)\ge V'(u_\eps*\eta_\e (t,\bar x)) \left[2c(t) v_\eps*\eta_\e(t,\bar x) - u_\eps(t,\bar x)(\partial_x v_\eps *\eta_\e(t,\bar x))\right].
\end{equation}
Integrating by parts we get
\begin{equation}\label{E_conv}
\begin{split}
\partial_x v_\eps *\eta_\e (t,\bar x) &=
 \int_0^\infty\partial_x v_\eps(t,\bar x + y)\eta_\e(-y)dy \\
 &=
c(t)\eta_\e(0) + \int_0^\infty v_\eps(t,\bar x + y)\eta_\e'(-y)dy.
\end{split}
\end{equation}
Plugging \eqref{E_conv} into \eqref{E_lower_der} we get
\begin{equation} \label{e:primastima}
\begin{split}
\partial_t v_\eps (t,\bar x)&\ge ~  V'(u_\eps*\eta_\e (t,\bar x))\left( -u_\eps(t,\bar x)c(t)\eta_\e(0) + \int_0^\infty v_\eps(t,\bar x+y)\left[ 2c(t)\eta_\e(-y)-u_\eps(t,\bar x)\eta_\e'(-y)\right]dy \right)\\
&= ~ -V'(u_\eps*\eta_\e (t,\bar x))\left( u_\eps(t,\bar x)c(t)\eta_\e(0) + \int_0^\infty v_\eps(t,\bar x+y)\left[ u_\eps(t,\bar x)\eta_\e'(-y) - 2c(t)\eta_\e(-y)\right]dy \right).
\end{split}
\end{equation}
Next, we combine~\eqref{e:D} with~\eqref{e:ee} and recall that $c(0) \leq L$ and that $\eta_\ee (y) = \eta (y/\ee) / \ee$. We obtain 
\begin{equation}\label{E_e_small}
u_\eps(t,\bar x)\eta_\e'(-y) - 2c(t)\eta_\e(-y) \ge 0 \quad \text{for a.e. $y \in [0, + \infty[$ at $t=0$}. 
\end{equation}
We now introduce the value $\tau \in [0, T]$ by setting  
\be \label{e:tau}
   \tau: = \sup \left\{ t \in [0, T]: \;  u_\eps(s,\bar x)\eta_\e'(-y) - 2c(s)\eta_\e(-y) \ge 0   \; \, \text{for a.e. $y \in [0, + \infty[$\, and every $s \in [0,t]$} \right\}.
\eq
Owing to~\eqref{e:c}, $v_\eps(t,\bar x +y)\ge -c(t)$ for every $y\ge 0$. 
By using~\eqref{e:primastima} and~\eqref{e:tau} we get that for every $t \in [0, \tau]$ 
\begin{equation}\label{E_est_low}
\begin{split}
\partial_t v_\eps (t,\bar x)&\ge   -V'(u_\eps*\eta_\e (t,\bar x))\left( u_\eps(t,\bar x)c(t)\eta_\e(0) -c(t)\int_0^\infty \big( u_\eps(t,\bar x)\eta_\e'(-y) - 2c(t)\eta_\e(-y)\big) dy \right) \\
&\stackrel{\int \eta (y) dy =1}{=}  -V'(u_\eps*\eta_\e (t,\bar x))( 2c(t)^2 ) \\
&\stackrel{\eqref{e:comeV}}{\ge}  2\delta c(t)^2.
\end{split}
\end{equation}
We now point out that $c(t) =  \Lip^-(S^\ee_t u_0)$ and hence, by property iii) in the statement of Theorem~\ref{T_BS}, it is a.e. differentiable. By combining~\eqref{e:c} and~\eqref{E_est_low} we get $\dot{c}(t) \leq 
- 2\delta c(t)^2$ and by a classical comparison argument for ODEs we arrive at
$$
    c(t) \leq  \frac{L}{2\delta Lt+1} \quad  \text{on $[0, \tau]$}. 
$$
To conclude, we are left to show that $\tau =T$. Assume by contradiction that $\tau <T$, then by the continuity of $c$ we get that
$$
    c(\tau) = u_\ee (\tau, \bar x) \inf_{y \in \mathrm{supp} \, \eta_\ee} \frac{\eta'_\ee (y)}{2 \eta_\ee (y)}
    \stackrel{\eqref{e:D},\eqref{E_MP}}{\ge} \frac{ \inf u_0 D }{ 2 \ee}  \stackrel{\eqref{e:ee}}{>} L \ge c(0).
$$
On the other hand, the inequality $\dot{c}(t)  \leq  - 2\delta c(t)^2$ on $[0, \tau[$ implies that $c(\tau)\leq c(0)$, which contradicts the previous chain of inequalities, shows that $\tau=T$ and hence establishes~\eqref{E_Lip} in {\sc Case 1}. \\
{\sc Case 2.} We define $\bar t \in [0, T]$ by setting 
\be 
     \bar t : = \inf \{ t \in [0, T]:  \;  \partial_x u_\eps(t,x)>0 \; \text{for every $x\in \R$} \}.
\eq
Assume $\bar t>0$: on the interval $[0, \bar t[$ we can apply the same argument as in {\sc Case 1} and, by using the continuity of the function $t\mapsto \Lip^-u_\eps(t)$ (see property iii) in the statement of Theorem~\ref{T_BS}),  establish~\eqref{E_Lip} on $[0, \bar t]$. Next, we use the fact that~\eqref{e:nl} preserves the monotonicity of the initial datum, see~\cite{BlandinGoatin,KeimerPflug2} and Remark~\ref{r:pm}. This implies that, for every $ t \in ]\bar t, T]$, $u_\ee(t, \cdot)$ is a monotone increasing function, that is $\Lip^-u_\eps(t) \leq 0$. If $\bar t=0$, then we can directly apply the preservation of monotonicity argument. This concludes the proof of Theorem~\ref{P_one_side}.
\qed

\begin{remark}
In the proof of Proposition \ref{P_one_side} we have used an approximation argument on the initial datum, because the computations require  that $u_\ee(t)\in C^2(\R)$, that is $u_0 \in C^2 (\R)$. Another possibility is to apply an approximation argument on the equation. More precisely, one could consider the viscous equation 
\begin{equation}
\partial_tu^\nu_\ee + \partial_x(u^\nu_\ee V(u^\nu_\ee *\eta_\e))=\nu \partial_{xx}u^\nu_\ee, \quad \nu >0,
\end{equation}
which has a regularizing effect. 
The same proof as in Proposition \ref{P_one_side} establishes the main estimate \eqref{E_est_low} for $u^\nu_\eps$.  Next, one could argue as in the proof of Corollary~\ref{cor:convergence} and show that $u^\nu_\ee$ strongly convergence in $L^1_{\mathrm{loc}}$ to $u_\ee$ as $\nu \to 0^+$.  By the $L^1$-lower semicontinuity of the map $u_\ee \mapsto \mathrm{Lip}^- (u_\ee)$ this eventually  yields~\eqref{E_est_low}. 
\end{remark}

\begin{remark}
Note that the one-sided Lipschitz estimate \eqref{E_Lip} does not depend on $L$, that is on $\Lip^-u_0$. However, we have only established \eqref{E_Lip}  for $\ee$ satisfying~\eqref{e:ee}, and hence the range of $\e>0$ such that  \eqref{E_Lip}  holds true does depends on $\Lip^-u_0$. This is the reason why Theorem~\ref{P_one_side} does not apply to general $BV$ initial data. 
\end{remark}
\subsection{Proof of Corollary~\ref{cor:convergence}}
We proceed according to the following steps. \\
{\sc Step 1:} uniform $BV$ bounds. Theorem \ref{P_one_side} implies that for every $\e>0$ sufficiently small and for every $t>0$ we have 
\begin{equation}\label{E_Lip_est}
\Lip^-u_\e(t)\le \frac{L}{2\delta tL+1}\le L.
\end{equation}
We now want to establish uniform bounds in $BV_{\mathrm{loc}}$, that is we want to show that, for any $R>0$, the quantity $\TV\{ u_\e(t);]-R,R[\}$ is uniformly bounded with respect to $\ee$ and $t$. We recall that 
\be \label{e:tv}
\TV\{ u_\e(t);[-R,R]\}= \sup_{-R \leq x_1 \leq \dots \leq x_N \leq R} \sum_{i=1}^{N-1} |u_\ee (t, x_{i+1} ) - u_\ee (t, x_i)|.
\eq
We consider separately the positive and the negative parts of the total variation of $u_\e(t)$, defined respectively by
\be
\begin{split}
\TV^+\{ u_\e(t);[-R,R]\}= &~ \sup_{-R \leq x_1 \leq \dots \leq x_N \leq R} \sum_{i=1}^{N-1} (u_\ee (t, x_{i+1} ) - u_\ee (t, x_i))^+, \\ 
\TV^-\{ u_\e(t);[-R,R]\}= &~ \sup_{-R \leq x_1 \leq \dots \leq x_N \leq R} \sum_{i=1}^{N-1} (u_\ee (t, x_{i+1} ) - u_\ee (t, x_i))^-.
\end{split}
\eq
From \eqref{E_Lip_est} it follows that
\be\label{E_TV-}
\TV^-\{ u_\e(t);[-R,R]\}\le 2LR,
\eq
therefore
\be\label{E_TV+}
\begin{split}
\TV^+\{ u_\e(t);[-R,R]\} =&~ \TV^-\{ u_\e(t);[-R,R]\} + u_\e(t,R)-u_\e(t,-R) \\
\le & ~ 2LR + u_{\max}.
\end{split}
\eq
It follows from \eqref{E_TV-} and \eqref{E_TV+} that
\begin{equation}\label{E_TV_loc}
\begin{split}
\TV\{ u_\e(t);[-R,R]\}= & ~\TV^+\{ u_\e(t);[-R,R]\} + \TV^-\{ u_\e(t);[-R,R]\} \\
\le & ~ 4LR + u_{\max}
\end{split}
\end{equation} 
and this concludes Step 1. \\
{\sc Step 2:} $\ee \to 0^+$ limit. First, we point out that, by the properties of convolution, from~\eqref{E_TV_loc}
we deduce that $u_\ee \ast \eta_\ee$ satisfies the same estimate. By using equation~\eqref{e:nle} we conclude that
$u_\e\in \Lip(\R_+,L^1_{\loc}(\R))$ and that the Lipschitz constant is uniform in $\ee$, provided $\e>0$ 
is sufficiently small. We apply the Helly-Kolmogorov Compactness Theorem and conclude that for every sequence $\e_n\to 0$ there exists a subsequence
$\e_{n_k}$ such that $u_{\e_{n_k}}$ converges to some function $\tilde u$ in $L^1_\loc(\R_+\times \R)$. Note that $\tilde u$ is a weak solution of the Cauchy problem~\eqref{e:cl},\eqref{E_initial}. To conclude the proof we are left to show that $\tilde u$ is actually \emph{the} entropy admissible solution of~\eqref{e:cl},\eqref{E_initial}. First, we point out that the flux function $u \mapsto u V(u)$ satisfies $(uV)''\leq - 2 \delta$. Next, we use~\eqref{E_Lip} and conclude that $\tilde u$ satisfies the Ole\u{\i}nik estimate 
\begin{equation}
\Lip^-\tilde u(t)\le \frac{1}{2\delta t}
\end{equation} 
and, owing to Chapter 8.5 in \cite{Dafermos_book}, this implies that $\tilde u$ is the entropy admissible solution.
\qed

\section{Proof of Theorem~\ref{t:blowup}} \label{s:pbu}
The proof of Theorem~\ref{t:blowup} is based on the explicit construction of an initial datum $u_0$ satisfying the statement.  To highlight the basic ideas of the construction and avoid some technicalities, we first provide in~\S\ref{ss:pp} the proof of Proposition~\ref{p:blowup} below. Proposition~\ref{p:blowup} is basically a weaker version of Theorem~\ref{t:blowup} as it establishes the total variation blow-up~\eqref{e:blowup} in the case where $\eta(x) =  \mathbbm{1}_{[-1, 0]}(x)$. Also, the initial datum $u_0$ constructed in the proof of Proposition~\ref{p:blowup} satisfies $0 \leq u_0 \leq 1$ and ${\TV \, u_0 < + \infty}$, but does not satisfy the one-sided Lipschitz condition $\mathrm{Lip}^- u_0 <+ \infty$. Next, in~\S\ref{ss:pt} we complete the proof of Theorem~\ref{t:blowup} by extending the construction to more general kernels and to initial data satisfying the one-sided Lipschitz condition. 
\begin{proposition} \label{p:blowup}
Assume that $V(u) = 1-u$ and $\eta(x) =  \mathbbm{1}_{[-1, 0]}(x)$, then there is $u_0 \in L^1 (\R)$ such that 
$0 \leq u_0 \leq 1$, ${\TV \, u_0 < + \infty}$ and the solution of the Cauchy problem~\eqref{e:nle},\eqref{E_initial} satisfies~\eqref{e:blowup}. 
\end{proposition} 
\subsection{Proof of Proposition~\ref{p:blowup}}\label{ss:pp}
Note that, under the assumptions of Proposition~\ref{p:blowup}, equation~\eqref{e:nl} boils down to
\be 
\label{e:nlee}
      \partial_t u_\ee + \partial_x \left[ u_\ee \left( 1 - \frac{1}{\ee} \int_{x}^{x+ \ee} u_\ee (t, z) \, dz \right) \right]=0.
\eq
The proof of Proposition~\ref{p:blowup} is organized as follows: in~\S\ref{s:cl} we use the same approach as in~\cite{CLM,KeimerPflug} and we discuss the characteristic lines of~\eqref{e:nlee}, in~\S\ref{ss:increase} we describe the basic idea underpinning the construction of $u_0$, in~\S\ref{s:id} we provide the actual construction of $u_0$, in \S\ref{s:pr} we establish some preliminary results and in~\S\ref{s:bu} we eventually conclude the proof. 
\subsubsection{Characteristic lines} \label{s:cl}

We refer to the analysis in Crippa and L{\'e}cureux-Mercier~\cite{CLM} and Keimer and Pflug~\cite{KeimerPflug} and we recall that the solution of~\eqref{e:nlee} given by Proposition \ref{T_BS}  can be obtained via a fixed point argument by considering the continuity equation 
\be \label{e:continuity}
    \partial_t u_\ee + \partial_x [u_\ee (1- w_\ee)] =0
\eq
requiring that the field $w_\ee$ be given by 
\be \label{e:v}
     w_\ee (t,x) = \frac{1}{\ee}
     \int_{x}^{x+\ee} u_\ee (t, z) \,dz \quad\implies\quad - \partial_x \big[ 1-w_\ee (t, x) \big] = \frac{u_\ee (t, x + \ee)- u_\ee (t, x)}{ \ee}.  
\eq 
The solution of~\eqref{e:continuity} can be expressed by relying on the method of characteristics. 
In the following we term $X_\ee(\cdot, y)$ the characteristic line 
starting at the point $y$, i.e. the solution of the Cauchy problem
\be \label{e:chardef}
\left\{
\begin{array}{l}
\displaystyle{\frac{d}{dt} X_\ee (t, y) = 1 -w_\ee (t, X_\ee(t,y)) }  \\ \\ 
X_\ee (0, y) = y. 
\end{array}
\right.
\eq
By~\eqref{e:v} we get that, if the initial datum is bounded (and hence the solution is bounded at all times, by the analysis in~\cite{CLM,KeimerPflug}), then for any fixed~$\ee>0$ the vector field $1-w_\ee$ is locally Lipschitz continuous with respect to the variable $x$ and continuous with respect to the variable $t$. This implies that the Cauchy problem~\eqref{e:chardef} is well posed and that the characteristic lines are well defined. Also, by combining~\eqref{e:continuity} with~\eqref{e:v} we get that the material derivative satisfies 
\be 
\label{e:material} 
    \frac{d}{dt} u_\ee (t, X_\ee) = - u_\ee(t,X_\ee) \partial_x \big[ 1-w_\ee (t,X_\ee)\big] = 
    u_\ee  (t, X_\ee) \ \frac{u_\ee (t, X_\ee+ \ee )- u_\ee (t, X_\ee)}{ \ee}.
\eq 
We point out in passing that formula~\eqref{e:material} formally shows that, at a maximum point of $u_\ee (t, \cdot)$, the material derivative is negative, which yields \eqref{E_MP}. 

\subsubsection{The mechanism for the increase of the total variation}\label{ss:increase}
Before entering into the technical details of the construction of the initial datum $u_0$ that triggers the blow-up of the total variation, we make some heuristic considerations to describe the basic ideas underpinning the construction of $u_0$. In particular, we describe the very basic mechanism that leads to the total variation increase.

\begin{figure}
\centering
\def\svgwidth{0.6\columnwidth}
\input{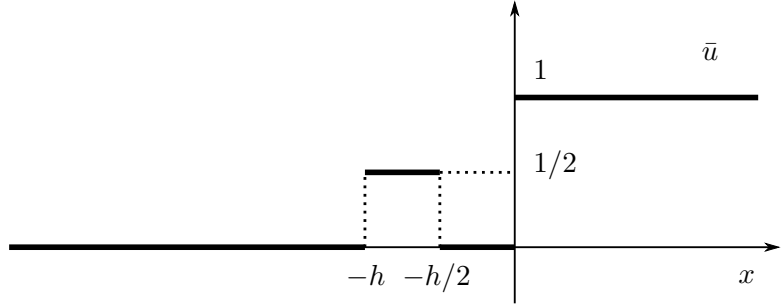}
\caption{The initial datum $\bar u$ triggering the total variation increase.}\label{f:bm}
\end{figure}
Fix $h>0$ and consider the function (sketched in~Figure~\ref{f:bm})
\begin{equation}
\label{e:baru}
   \bar u (x) = 
   \left\{
   \begin{array}{ll}
   1/2 & x \in [-h, -h/2] \\
   1 & x \ge 0 \\
   0 & \text{otherwise.} \\
   \end{array}
   \right.
\end{equation}

Consider now the solution of the Cauchy problem obtained by coupling~\eqref{e:nlee} with the initial 
condition~$u_\ee(0, x) = \bar u(x)$ in~\eqref{e:baru}. We observe that: 
\begin{itemize}
\item[(a)] $u_\ee (t, x) \equiv 1$ if $x\ge0$ and $t\ge 0$. Loosely speaking, this can can be seen by combining two facts: (i)~the nonlocal term evaluated at the point $(t, x)$ is only affected by the values of $u_\ee(t, z)$ at~$z\ge x$ and (ii) the characteristic line 
starting at $x=0$ has zero speed and hence information cannot cross the vertical axis. This implies that the values of the solution $u_\ee$ on $\R_+ \times \R_+$ are only affected by the values of the initial datum $\bar u$ on $\R_+$. Since $\bar u\equiv 1$ on $\R_+$, then $u_\ee \equiv 1$ on~$\R_+ \times \R_+$. 
\item[(b)] Assume that $\ee>h$ and consider the characteristic lines starting at $y \in [-h, -h/2]$. Since $\bar u (y + \ee) =1$, owing to~\eqref{e:material}, the material derivative at $t=0$ satisfies 
$$
      \left. \frac{d}{dt} u_\ee (t, X_\ee(t,y)) \right|_{t=0} =
    u_\ee(0,y) \ \frac{1- u_\ee (0, y)}{ \ee} = \frac{1}{4 \ee} >0, 
$$
which means that, at least for a small time, $u_\ee$ increases along the characteristic line $X_\ee(t, y)$. 
\item[(c)] By using again~\eqref{e:material}, we see that, if $\bar u(y) =0$, then $u_\ee$ is identically $0$ along the characteristic line~$X_\ee (\cdot, y)$.  
\end{itemize}
As a consequence we have that, for $\ee > h$, the solution $u_\ee$ is identically equal to $1$ on $\R_+ \times \R_+$, increases (locally in time) along the characteristic lines $X_\ee (\cdot, y)$ if $y \in [-h, -h/2]$, and vanishes identically elsewhere. We can infer 
that 
$$
   \mathrm{Tot Var} \, u_\ee (\tau, \cdot) >   \mathrm{Tot Var} \, \bar  u =2, \quad \text{for every $\tau >0$ sufficiently small}. 
$$

\subsubsection{Construction of the initial datum $u_0$}\label{s:id} 
There are two main issues we have to address in order to construct an initial datum as in the statement of Proposition~\ref{p:blowup}: (i)~in~\S\ref{ss:increase} the total variation increases only if $\ee > h$, and (ii)~we claim that the total variation not only increases but actually blows up. To tackle these issues, we introduce the building block $a: \R \to \R$ by setting 
\be 
\label{e:bb}
      a (x) : = \mathbbm{1}_{[-1, -3/4]}(x)
\eq
and we define $u_0$ as 
\be 
\label{e:uzero}
   u_0(x) : = \mathbbm{1}_{[0, + \infty[} (x) + \sum_{k=0}^{\infty} 2^{-k} a ({2^k} x). 
\eq
See Figure~\ref{f:u0} for a representation. 
\begin{figure}
\centering
\def\svgwidth{0.4\columnwidth}
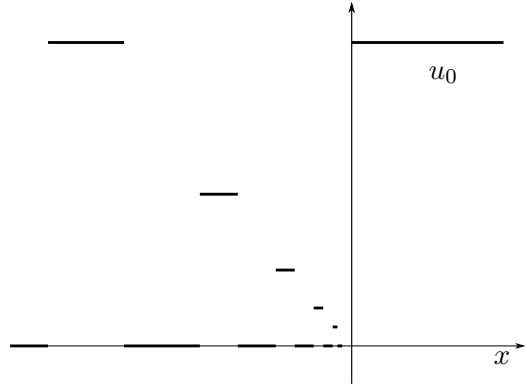
\caption{The initial datum $u_0$ triggering the total variation blow-up.} \label{f:u0}
\end{figure}
Note that due to the chosen scaling the building blocks of $u_0$ do not overlap and are separated by intervals where $u_0=0$. 
This implies that 
\be
\label{e:mptv}
  0 \leq u_0(x) \leq 1 \;  \text{for a.e. $x \in \R$}, \quad 
  \mathrm{TotVar} \ u_0 = 1 + 2 \sum_{k=0}^{\infty} 2^{-k}= 5 < + \infty.
\eq
Note furthermore that, very loosely speaking, $u_0$ is made by a sequence of building blocks that approaches the ``big jump" located at $t=0$. In this way, for every $\ee >0$ there are infinitely many building blocks that behave as the initial datum $\bar u$ in the example of \S\ref{ss:increase}. Each of them contributes to the total variation increase and this is the basic mechanism that leads to the total variation blow-up.   
\subsubsection{Preliminary results}
\label{s:pr}
In this paragraph we establish some qualitative properties of the solution of the Cauchy problem obtained by coupling~\eqref{e:nlee} with the initial datum $u_0$ in~\eqref{e:uzero}.

By combining the first inequality in~\eqref{e:mptv} with the maximum principle~\eqref{E_MP} we get that 
\be 
\label{e:mpfin}
    0 \leq u_\ee (t, x) \leq 1,  \;  \text{for a.e. $(t, x) \in \R_+ \times \R$}
\eq
and by recalling~\eqref{e:v} we arrive at 
\be 
\label{e:mpvf}
    0 \leq 1- w_\ee (t, x) \leq 1,  \;  \text{for a.e. $(t, x) \in \R_+ \times \R$}. 
\eq

\begin{lemma} \label{l:key}
Let $u_\ee$ be the solution of~\eqref{e:nlee} with initial datum~\eqref{e:uzero}, then 
\be 
 \label{e:u1}
     u_\ee (t, x) =1, \quad \text{for a.e. $(t, x)$ such that $x \ge 0$ and $t \ge 0$}.
\eq
\end{lemma}
Lemma~\ref{l:key} can be shown arguing as in item~(a) in \S\ref{ss:increase} and its proof exploits the fact that, owing to the particular expression of the velocity field~\eqref{e:v}, the values of $u_\ee$ on $\R_+ \times \R_+$ are only affected by the values of the initial datum on $\R_+$ and hence on $\R_+ \times \R_+$ the solution $u_\ee$ behaves ``as if the initial datum is the constant $1$". This formal argument can be turned in a rigorous proof by 
a fixed-point argument as in~\cite{CLM,KeimerPflug}. 
As a consequence of Lemma~\ref{l:key}, we get the following fact. 
\begin{lemma}
\label{l:vainla}
Let $u_\ee$ be the solution of~\eqref{e:nlee} with initial datum~\eqref{e:uzero}, then 
\be \label{e:charpiccola} 
 y \leq X_\ee(t, y) \leq 0, \quad \text{for every $t\ge 0$, $y \leq 0$ and $\ee >0$}. 
\eq  
\end{lemma}
\begin{proof}
By combining~\eqref{e:chardef} with the fact that
$w_\ee(t, x) = 1$ for every $x \ge 0$ and $t \ge 0$
we get that
\be \label{e:characteristic0} 
X_\ee (t, 0) =0, \quad \text{for every $\ee>0$ and $t \ge 0$}. 
\eq
Since the characteristic lines cannot intersect, this implies the inequality $X_\ee(t, y) \leq 0$ in~\eqref{e:charpiccola}. 
The inequality $y \leq X_\ee(t, y)$ follows from the the first inequality in~\eqref{e:mpvf}. 
\end{proof}
\subsubsection{Total variation blow-up} \label{s:bu}
We can now conclude the proof of Proposition~\ref{p:blowup}. 
Fix $\ee \in ]0, 1[$. By combining~\eqref{e:charpiccola} and~\eqref{e:u1} we get that 
\be
     u_\ee \big(t, X_\ee (t, y) + \ee \big)=1, \quad \text{for every $y \in [-\ee, 0]$ and $t \ge 0$}. 
\eq
Owing to~\eqref{e:material}, the material derivative satisfies 
\be \label{e:ode} 
    \frac{d}{dt} u_\ee (t, X_\ee) = u_\ee (t, X_\ee) \ \frac{1- u_\ee(t, X_\ee)}{\ee}, \quad \text{for every $y \in [-\ee, 0]$ and $t \ge 0$}.
\eq
By explicitly computing the solution of the ODE~\eqref{e:ode} we arrive at 
\be \label{es}
     u_\ee (t, X_\ee(t, y)) 
     = \frac{u_0(y)}{[1- u_0(y)]e^{-t/\ee} + u_0(y)}, \quad \text{for every $y \in [-\ee, 0]$ and $t \ge 0$}.
\eq
We recall~\eqref{e:uzero} and notice that 
\begin{equation}\label{e:osc}
u_0(y) = 
\left\{\begin{array}{ll}
2^{-k} & \text{if~$y \in [- 2^{-k}, -2^{-k}\cdot 3/4]$ for some $k \in \mathbb N$} \phantom{\displaystyle{\int}}\\
0 & \text{if~$y \in [-2^{-k}\cdot 3/4, -2^{-(k+1)}]$ for some $k \in \mathbb N$. } \phantom{\displaystyle{\int}}
\end{array}\right.
\end{equation} 
Using the fact that characteristic lines cannot intersect we conclude that 
\be 
\label{e:key}
    \mathrm{TotVar} \, u_\ee (\tau, \cdot) \ge 2 \sum_{k \ge - \log_2  \ee } \frac{2^{-k}}{[1- 2^{-k}]e^{-\tau/\ee} + 2^{-k}},
\eq
where we have used~\eqref{e:osc} and the restriction $k \ge - \log_2 \ee$ in the sum is due to the fact that~\eqref{es} is valid for $y \in [-\ee, 0]$. 

To establish~\eqref{e:blowup} it now suffices to show that, for every $\tau>0$, the right hand side in~\eqref{e:key} is not 
bounded as~$\ee \to 0^+$.  To this end, we first point out that 
\be 
\label{e:el1}
     \frac{2^{-k}}{[1- 2^{-k}]e^{-\tau/\ee} + 2^{-k}} \ge \frac{1}{2} \quad\Leftrightarrow\quad
     k \leq - \log_2 \left( \frac{e^{-\tau/\ee}}{1 + e^{-\tau/\ee}} \right),  
\eq
which owing to~\eqref{e:key} yields 
\be 
\label{e:key2}
    \mathrm{TotVar} \, u_\ee (\tau, \cdot) \ge \sharp 
    \left\{  k \in \mathbb N: \;  - \log_2 \ee \leq k \leq  - \log_2 \left( \frac{e^{-\tau/\ee}}{1 + e^{-\tau/\ee}} \right)   \right\}.
\eq     
In the previous expression, the symbol $\sharp$ denotes the cardinality of a set. By plugging the elementary inequality
$$
   - \log_2 \left( \frac{e^{-\tau/\ee}}{1 + e^{-\tau/\ee}} \right)  \ge - \log_2 \left( e^{-\tau/\ee} \right) = \frac{\tau}{\ee} \log_2 e 
$$
into~\eqref{e:key2} and choosing $\ee = 2^{-j}$, we get that 
\be 
\label{e:key3}
    \mathrm{TotVar} \, u_{2^{-j}} (\tau, \cdot) \ge \sharp 
    \left\{  k \in \mathbb N: \;   j \leq k \leq  2^j \tau  \log_2 e  \right\}.
\eq   
For any given $\tau>0$, the right hand side of~\eqref{e:key3} blows up as $j \to + \infty$, 
yielding~\eqref{e:blowup}. This concludes the proof of Proposition~\ref{p:blowup}. \qed
\subsection{Conclusion of the proof of Theorem~\ref{t:blowup}} \label{ss:pt}
To complete the proof of Theorem~\ref{t:blowup} we are left to show that i) we can modify the construction of $u_0$ in such a way that it satisfies the condition $\mathrm{Lip}^- u_0 < \infty$, and ii) we can extend the blow-up proof to the case of more general convolution kernels. 

To tackle issue i), it suffices to replace the building block $a$ in~\eqref{e:bb} with 
\be  \label{e:newa}
    \tilde a(x) : = 
    \left\{
     \begin{array}{ll}
     0 & x < -1 \\
     - 4 x -3  & -1 \leq x < -3/4 \\
    0 & x \ge -3/4. \\    
     \end{array}
    \right.
\eq
We define $u_0$ by plugging the above espression into~\eqref{e:uzero} and obtain that $u_0$ satisfies $0 \leq u_0 \leq 1$, $\TV \, u_0 \leq 5$ and $\mathrm{Lip}^- u_0 = 4$. One can then study the evolution of $u_\ee$ along the characteristic lines $X_\ee (\cdot, y)$ with $y = 2^{-k}$, $k \in \mathbb N$ and conclude that the key estimate~\eqref{e:key} is still valid. The rest of the proof of Proposition~\ref{p:blowup} extends with no need of modifications. 

To tackle issue ii) (extension of the proof to the case of more general kernels) we fix a Lipschitz continuous kernel $\eta$ as in the statement of Theorem~\ref{t:blowup}. We go back to the discussion about characteristic lines in~\S\ref{s:cl} and we point out that we have to replace~\eqref{e:v} with
\be \label{e:v2}
     w_\ee (t, x) = \int_x^{+\infty} u_\ee (y) \eta_\ee (x-y) dy \implies 
    - \partial_x [1 - w_\ee (t, x) ] = - u_\ee (x) \eta_\ee (0) + \int_x^{+\infty} u_\ee (y) \eta'_\ee (x-y) dy. 
\eq
Note that, since $\eta_\ee$ has unit integral and $0\leq u_\ee \leq 1$, then~\eqref{e:mpvf} is still valid and the rest of the analysis in~\S\ref{s:pr} extends with no modifications.  We now discuss how we can modify~\eqref{e:material}. Let us fix $x^\ast<0$ such that
\be
\label{e:xast}
   M:=\frac{\eta(x^\ast)}{\eta(0)} \ge \frac{3}{4} . 
\eq 
Next, we fix $y$ such that $ \ee x^\ast \leq y < 0$, which owing to~\eqref{e:charpiccola} yields $- \ee x^\ast\leq X_\ee (t, y)<0$ for every $t\ge0$. We have the following chain of inequalities:
\be \begin{split}\label{e:K}
     \frac{d }{dt} u_\ee (t, X_\ee)& \stackrel{\eqref{e:v2}}{=} u_\ee (t, X_\ee) \left(  - u_\ee (t, X_\ee) \eta_\ee (0) + \int_{X_\ee}^{+\infty} u_\ee (y) \eta'_\ee (X_\ee-y) dy \right) \\ & \stackrel{u_\ee, \eta_\ee' \ge 0}{\ge}
    u_\ee  (t, X_\ee) \left(  - u_\ee  (t, X_\ee) \eta_\ee (0) + \int_{0}^{+\infty} \underbrace{u_\ee (y)}_{=1 \text{by~\eqref{e:u1}}} \eta'_\ee (X_\ee-y) dy \right) \\ &
     = 
      u_\ee  (t, X_\ee) \left(  - u_\ee  (t, X_\ee) \eta_\ee (0) + \int_{0}^{+\infty}  \eta'_\ee 
(X_\ee-y) dy \right) 
    \\ & =  u_\ee  (t, X_\ee) \big(  - u_\ee  (t, X_\ee) \eta_\ee (0) +  \eta_\ee 
(X_\ee) \big) \stackrel{\eta_\ee'\ge 0}{\ge} 
      u_\ee  (t, X_\ee) \big(  - u_\ee  (t, X_\ee) \eta_\ee (0) +  \eta_\ee 
(x^\ast \ee) \big).
\phantom{\int}
\end{split}
\eq
We recall that the constant $M$ is defined in~\eqref{e:xast}.  Since $\eta_\ee (x) = \eta(x/\ee)/\ee$, then
$M= \eta_\ee 
(x^\ast \ee)/ \eta_\ee(0)$ and 
 by using~\eqref{e:K} we get 
$$
    \frac{d }{dt} u_\ee (t, X_\ee) \ge  u_\ee  (t, X_\ee) \eta_\ee (0) \left(  - u_\ee  (t, X_\ee) + M \right)
    = \eta (0) \frac{u_\ee  (t, X_\ee)  \left(  - u_\ee  (t, X_\ee) + M \right)}{\ee}.
$$
We compute the explicit solution of the ODE $\dot{u}= \eta(0) u (M-u)/\ee$ and by a classical comparison argument for ODEs we conclude that
\be 
    u_\ee (t, X_\ee (t, y)) \ge \frac{M u_0(y)}{(M-u_0(y))e^{-\eta(0)t/\ee } + u_0(y)}
   \quad \text{for every $y \ge  \ee x^\ast$}. 
\eq
This implies that we can replace~\eqref{e:key} with 
\be \label{e:keynew}
       \mathrm{TotVar} \, u_\ee (\tau, \cdot) \ge 2 \sum_{k \ge - \log_2 \ee x^\ast} 
        \frac{M 2^{-k}}{(M- 2^{-k})e^{-\eta(0)t/\ee } + 2^{-k}}.
\eq  
 The rest of the analysis in~\S\ref{s:bu} straightforwardly extends and this yields~\eqref{e:blowup}. 
 \qed
 
\section*{Acknowledgments} \noindent
MC is partially supported by
the Swiss National Science Foundation grant 182565.
GC and EM are partially supported by the Swiss National Science Foundation
grant 200020\_156112 and by the ERC Starting Grant 676675 FLIRT. LVS is a member of the GNAMPA group of INDAM. Part of this work was done when MC and LVS were visiting the University of Basel: its  kind hospitality is gratefully acknowledged.  
\bibliographystyle{plain}
\bibliography{tv}
\end{document}